\documentclass[oneside,abstracton]{scrartcl}
\usepackage[utf8]{inputenc}
\usepackage[T1]{fontenc}
\usepackage{lmodern}
\usepackage[english]{babel}

\usepackage{amsmath}
\usepackage{amsthm}
\usepackage{amssymb}
\usepackage[nobysame]{amsrefs}
\usepackage{enumerate}
\usepackage{cite}
\usepackage{mleftright}

\swapnumbers
\mleftright

\theoremstyle{plain}
	\newtheorem{theorem}{Theorem}[section]
	\newtheorem*{theorem*}{Theorem}
	\newtheorem{lemma}[theorem]{Lemma}
	\newtheorem{corollary}[theorem]{Corollary}
\theoremstyle{remark}
	\newtheorem{remark}[theorem]{Remark}
	
\newcommand{\cK}	{\mathcal K}

\newcommand{\cP}	{\mathcal P}
\newcommand{\cU}	{\mathcal U}
\newcommand{\cZ}	{\mathcal Z}

\newcommand{\R}		{\mathbb R}

\newcommand{\aff}{\operatorname{aff}}
\newcommand{\conv}{\operatorname{conv}}
\newcommand{\sgn}{\operatorname{sgn}}

\title{A Valuation-Theoretic Approach to Translative-Equidecomposability}
\author{Katharina Kusejko and Lukas Parapatits}
\date{}

\begin{document}
  \maketitle

  \begin{abstract}
All simple translation-invariant valuations on polytopes are classified.
As a direct consequence the well-known conditions for translative-equidecomposability are recovered.
Furthermore, a simplified proof of the classification of \emph{continuous} simple translation-invariant valuations is presented.
\\[0.5cm]

  \end{abstract}

  \section{Introduction}
The study of equidecomposability has always been closely connected to valuation theory.
In fact, Dehn's solution of Hilbert's Third Problem used valuations as the core ingredient.

Let $G$ be a subgroup of the group of motions that contains all translations.
Two polytopes in $\R^n$ are said to be $G$-equidecomposable if
they can be cut into finitely many pieces such that
there is a bijection between the two sets of pieces
and corresponding pieces are equal up to a transformation from $G$.

A valuation $\phi$ is a map from the set of polytopes to $\R$ such that
\[
  \phi(P \cup Q) = \phi(P) + \phi(Q) - \phi(P \cap Q)
\]
for all polytopes $P$ and $Q$ whenever $P \cup Q$ is also convex.

After Dehn's hallmark result a systematic study of valuations was initiated by Hadwiger \cite{Hadwiger1957} in the 1950's.
In recent years the interest in valuations has increased tremendously (see e.g.\
\cites{Alesker2001, AleskerBernig2012, AleskerBernigSchuster2011,
BernigFaifman, BernigFu2011, Klain2000, ParapatitsSchuster2012,
ParapatitsWannerer2013, Wannerer2014}).
Classification and characterization results have been a particular focus (see e.g.\
\cites{Abardia2012, AbardiaBernig2011, Alesker1999, AleskerFaifman2014, Bernig2009, Haberl2011, Haberl2012,
HaberlLudwig2006, HaberlParapatits2014_1, HaberlParapatits2014_2, HugSchneider2014, Ludwig2002_1,
Ludwig2002_3, Ludwig2003, Ludwig2005, Ludwig2006, Ludwig2010, LudwigReitzner2010,
Parapatits2014_1, Parapatits2014_2, Schuster2010, SchusterWannerer2012, Wannerer2011}).

One of the far reaching results of Hadwiger \cite{Hadwiger1952} (see also McMullen \cite{McMullen1983})
is a complete classification of weakly-continuous simple translation-invariant valuations.
Here, $\cU$ denotes the set of (orthonormal) frames and
$\cU^{n-k}_P$ denotes those frames that are $P$-tight and have $n-k$ entries.
See Section \ref{se: preliminaries} for precise definitions of the notation.

\begin{theorem*}[cf.\ Thereom \ref{th: classification dilation}]
  A map $\phi \colon \cP^n \to \R$ is a weakly-continuous simple translation-invariant valuation if and only if
  for all $U \in \cU$ there exists a constant $c_U \in \R$ such that $U \mapsto c_U$ is odd and
  \[
    \phi(P) = \sum_{k=1}^n \sum_{U \in \cU^{n-k}_P} c_U V_k(P_U)
  \]
  for all $P \in \cP^n$.
\end{theorem*}

Our main result generalizes this classification to simple translation-invariant valuations without any regularity assumption.

\begin{theorem*}[see Theorem \ref{Th2}]
  A map $\phi \colon \cP^n \rightarrow \R$ is a simple translation-invariant valuation if and only if
  for all $U \in \cU$ there exists an additive function $f_U \colon \R \to \R$ such that $U \mapsto f_U$ is odd and
  \[
    \phi(P) = \sum_{k=1}^n \sum_{U \in \cU^{n-k}_P} f_U(V_k(P_U))
  \]
  for all $P \in \cP^n$.
\end{theorem*}

Hadwiger's \cite{Hadwiger1957} \emph{formal main criterion} (in German: \emph{Formales Hauptkriterium}) establishes a connection
between the $G$-equidecomposability of two polytopes and simple $G$-invariant valuations.
It states that two polytopes are $G$-equidecomposable if and only if
their values agree for every simple $G$-invariant valuation.
Hence, it is possible to solve the problem of equidecomposability by establishing classification theorems of valuations.
However, so far this approach has not been successfully applied apart from special cases.
As a direct consequence of our main result we are now able to recover the following necessary and sufficient conditions for translative-equidecomposability.
They were first proved by Jessen and Thorup \cite{JessenThorup1978} and independently by Sah \cite{Sah1979}.
Before that, in dimension $n=2$ respectively $n=3$ the problem of translative-equidecomposability was already solved by
Glur and Hadwiger \cite{GlurHadwiger1951} respectively Hadwiger \cite{Hadwiger1968}.
See Section \ref{se: conditions} for precise definitions of the (basic) Hadwiger functionals $H_U$.

\begin{theorem*}[see Corollary \ref{co: condition for translative-equidecomposability}]
  Two elements $P$ and $Q$ of $\bigcup \cP^n_o$ are translative-equidecomposable if and only if
  \[
    H_U(P) = H_U(Q)
  \]
  for all $U \in \cU$.
\end{theorem*}

As we will see later on, it is also possible to recover the classification of simple translation-invariant valuations
from the conditions on translative-equidecomposability,
which makes these two theorems more or less equivalent (see Remark \ref{re: classification from conditions}).
However, in the opinion of the authors, the approach taken in the present paper yields a much simpler proof of the conditions on translative-equidecomposability.
Moreover, the techniques used in the present paper are more geometric as opposed to algebraic.

Finally, in Section \ref{se: Klain-Schneider} we will see how the techniques used in the proof of the classification above
can also be used to give a simple and simultaneous proof of the following result(s)
by Klain \cite{Klain1995} (see also \cite{KlainRota1997}) and Schneider \cite{Schneider1996}.
These two results turned out to be of great importance for later developments.
In particular, they were a crucial ingredient in the proof of Alesker's irreducibility theorem \cite{Alesker2001}.

\begin{theorem*}[see Theorem \ref{th: Klain-Schneider}]
  Let $n \geq 2$.
  A map $\phi \colon \cK^n \to \R$ is a continuous simple translation-invariant valuation if and only if
  there exist a constant $c \in \R$ and a continuous odd function $f \colon S^{n-1} \to \R$ such that
  \[
    \phi(K) = c V(K) + \int_{S^{n-1}} f \ dS_K
  \]
  for all $K \in \cK^n$.
  Furthermore, $f$ is unique up to restrictions of linear functions.
\end{theorem*}

  \section{Preliminaries} \label{se: preliminaries}
In this section we recall the most important definitions and results used later on.
As a general reference we refer the reader to the excellent book by Schneider \cite{Schneider2014}.
We will work in $\R^n$ with $n \geq 1$.
Although it is not strictly necessary, we will now fix an orthonormal basis $e_1,\ldots,e_n$ of $\R^n$ for notational convenience. 
Moreover, let $\left\langle  .,. \right\rangle$ denote the associated scalar product.
By $\conv$ and $\aff$ we denote the \emph{convex hull} and \emph{affine hull}, respectively.
A \emph{polytope} $P$ is the convex hull of finitely many points in $\R^n$ and by $\cP^n$ we denote the set of all polytopes in $\R^n$.
By $\bigcup \cP^n$ we denote the set of all subsets $Q$ of $\R^n$ which can be written as a finite union of polytopes, i.e.\ $Q = P_1 \cup \ldots \cup P_m$ for $P_1,\ldots,P_m \in \cP^n$.
Moreover, let $\cP^n_o$ denote the polytopes of full dimension $n$ and $\bigcup \cP^n_o$ their finite unions.
By $\cP^k(M)$ we denote the set of all elements in $\cP^n$ which lie in $M$, where $M$ is a $k$-dimensional (affine) subspace of $\R^n$.
For $k=n-1$, elements in $\cP^{n-1}(M)$ and maps on $\cP^{n-1}(M)$ are denoted by dashed capital and small letters, respectively.

A map $\phi \colon \cP^n \rightarrow \mathbb{R}$ is called a \emph{valuation} on $\cP^n$ if
\begin{equation}\label{valuation}
\phi(P \cup Q) = \phi(P) + \phi(Q) - \phi(P \cap Q)
\end{equation}
whenever $P,Q, P \cup Q \in \cP^n$. Note that for every valuation on $\cP^n$ there exists a unique extension to $\bigcup \cP^n$, which we will also denote by $\phi$, such that $\phi$ satisfies equation (\ref{valuation}) for all $P,Q \in \bigcup \cP^n$ (see e.g.\ \cite[Chapter 6.2]{Schneider2014}).

If 
$$ \phi(P) = 0,\ \forall \ P \in \cP^n \text{ with } \dim(P) \leq n-1,$$
we call $\phi$ \emph{simple}. 

If $P$ and $Q$ in $\cP^n_o$ have disjoint interiors, we write $P \sqcup Q$ for their union. For a simple valuation $\phi$, the valuation property (\ref{valuation}) becomes
\begin{equation}\label{eq: additivity}
  \phi(P \sqcup Q) = \phi(P) + \phi(Q) 
\end{equation}
which we will refer to as the \emph{additivity} of $\phi$.

Moreover, $\phi$ is called \emph{translation-invariant} on $\cP^n$ if
$$ \phi(P+t) = \phi(P)$$
for all $P \in \cP^n$ and $t \in \mathbb{R}^n$.

By Hadwiger \cite[Chapter 2.2.3]{Hadwiger1957}, we know that for all simple translation-invariant valuations $\phi$, we have the following representation.

\begin{theorem} \label{homparts}
Every simple translation-invariant valuation $\phi$ can be decomposed as a sum
\begin{equation} 
\phi= \sum_{i=1}^n \phi_i,
\end{equation}
where $\phi_i$ is a rational-$i$-homogeneous valuation, i.e.\ for all rational $\lambda > 0$ we have
$$ \phi_i(\lambda P) = \lambda^i \phi(P)$$
for all $P \in \cP^n$.
\end{theorem}
Note that a similar decomposition is also possible for translation-invariant valuations $\phi$, which are not necessarily simple. This representation is known as the \emph{McMullen decomposition} (see e.g.\ \cite[Chapter 6.3]{Schneider2014}). 

Given a polytope $P \in \cP^n_o$, we say that $P$ can be \emph{decomposed} into polytopes $P_1,\ldots,P_m \in \cP^n_o$, if 
$$ P =  P_1 \sqcup\ldots \sqcup P_m.$$
We call $P$ and $Q$ in $\cP^n_o$ \emph{translation-equivalent}, denoted by $P \simeq Q$, if there exists a $t \in \mathbb{R}^n$ such that $P+t=Q$. Moreover, $P$ and $Q$ are called \emph{translative-equidecomposable}, denoted by $P \sim Q$, if
$$ P = P_1 \sqcup \ldots \sqcup P_m,\ Q = Q_1 \sqcup \ldots \sqcup Q_m \text{ and } P_i \simeq Q_i \text{ for } i=1,\ldots,m. $$

In general, for any Euclidean group $G$ in $\mathbb{R}^n$ which contains the translation group, we say that two polytopes $P$ and $Q$ are \emph{G-equidecomposable} if
\[
  P = P_1 \sqcup \ldots \sqcup P_m,
\]
\[
  Q = Q_1 \sqcup \ldots \sqcup Q_m
\]
and for all $i=1,\ldots,m$ we have $P_i = g_i(Q_i)$ for some $g_i \in G$. Similarly, a map $\phi$ on $\cP^n$ is called $G$-invariant, if $\phi = \phi \circ g$ for all $g \in G$. By Hadwiger \cite[Chapter 2.2.5]{Hadwiger1957}, we know how to decide whether two given polytopes $P$ and $Q$ in $\cP^n$ are $G$-equidecomposable.

\begin{theorem}\label{FormalCriterion}
Two elements $P$ and $Q$ of $\bigcup \cP^n_o$ are $G$-equidecomposable if and only if
$$\phi(P)=\phi(Q),$$ 
for all simple $G$-invariant valuations.
\end{theorem}

Note that this result has only a formal character as long as we do not know how to describe all simple $G$-invariant valuations explicitly. In the following, we are mainly interested in translative-equidecomposable polytopes, i.e.\ we take $G$ to be the group of all translations.

For $P, Q \in \cP^n$ we define the \emph{Minkowski sum} $P+Q$ of $P$ and $Q$ by
$$ P +Q := \{ p + q \ | \ p \in P, \ q \in Q \}. $$
We call $P \in \cP^n_o$ a \emph{$k$-cylinder} if $P = P_1 + \ldots + P_k$, where $P_i \subseteq V_i$ for subspaces $V_i$ with $V_1 \oplus \ldots \oplus V_k = \mathbb{R}^n$ and $\dim(V_i)\neq0$. The set of all $k$-cylinders is denoted by $Z_k$. For a $P \in Z_k$ of dimension $n$ we therefore have $\sum_{i=1}^k \dim(P_i) = n$. Note that $Z_1 = \cP^n_o$.

By $\cZ_k$ we denote the set of all elements $P \in \bigcup \cP^n_o$ which are translative-equidecompos\-able to some $Q = Q_1 \sqcup \ldots \sqcup Q_m$ with $Q_i \in Z_k$, $1 \leq i \leq m$, together with the empty set. Since every $k$-cylinder is a $(k-1)$-cylinder as well, we have
$$ \cZ_n \subseteq \cZ_{n-1} \subseteq \ldots \subseteq \cZ_1.$$
Clearly, $\cZ_1 = \bigcup \cP^n_o$. Note that for $A,B \in \cZ_k$ also $A \sqcup B \in \cZ_k$. 

We call $P$ and $Q$ in $\cP^n_o$ \emph{translative-equidecomposable modulo $\cZ_k$}, denoted by 
$$ P \sim Q \mod  \cZ_k,$$
if there exist $A,B \in \cZ_k$ such that $P \sqcup A \sim Q \sqcup B$.

Recall the following result by Hadwiger \cite[Chapter 1.3.7]{Hadwiger1957}.

\begin{lemma} \label{modZk}
For $P \in \cZ_k$ and an integer $m>0$ we have
\begin{equation} \label{equivalence}
 m P \sim m^k \bullet P \mod \cZ_{k+1}.
\end{equation}
Note that $m P$ is the dilation of $P$ by $m$, whereas $m^k \bullet P$ denotes the union of $m^k$ disjoint translated copies of $P$.
\end{lemma}

For the dilation of simplices $P$ and $Q$ by some $\lambda > 0$ and $\mu >0$, respectively,
we need the following basic geometric lemma due to Jessen and Thorup (see \cite{JessenThorup1978}) and Sah (see \cite{Sah1979}).
In dimension $n=3$ it was already known to Hadwiger \cite{Hadwiger1968}.
For this, define the simplex
  $$[a_1,\ldots,a_k] := \conv \{0,a_1,a_1+a_2,\ldots,a_1+\ldots+a_k\}, $$
where $a_1,\ldots,a_k$ are linearly independent vectors.

\begin{lemma} \label{LemmaModZ3}
  For linearly independent vectors $a_1,\ldots,a_n$ and for $\lambda, \mu >0$ we have
  \begin{equation*}
    \lambda[a_1,\ldots,a_j] +  \mu[a_{j+1},\ldots,a_n] \sim \mu[a_1,\ldots,a_j] + \lambda[a_{j+1},\ldots,a_n] \mod \cZ_3,
  \end{equation*}
  for every $j \in \{1,\ldots,n\}$.
\end{lemma} 

In the following we diverge slightly from the notation of the book by Schneider \cite{Schneider2014}.
Let $\cU^k$ denote the set of all (orthonormal) $k$-frames in $\mathbb{R}^n$, i.e.\ the set of all ordered $k$-tuples of orthonormal vectors in $\mathbb{R}^n$ for $k=0,\ldots,n-1$. Note that $\cU^0$ contains the empty tuple $()$ only.
For $P \in \cP^n$ the face $P_U$ for $U=(u_1,\ldots,u_k) \in \cU^k$ is defined inductively by $P_{()}=P$ and $P_{(u_1,\ldots,u_k)} = F(P_{(u_1,\ldots,u_{k-1})},u_k)$, where $F(P_{(u_1,\ldots,u_{k-1})},u_k)$ denotes the face of $P_{(u_1,\ldots,u_{k-1})}$ in direction $u_k$.
A frame $U=(u_1,\ldots,u_k) \in \cU^k$ is called \emph{$P$-tight}, if $\dim (P_{(u_1,\ldots,u_r)})=n-r$ for $r=0,\ldots,k$. Clearly, for every $P$ there are only finitely many $P$-tight frames, which are denoted by $\cU^k_P$. In particular, for $\dim (P)<n$ there are no $P$-tight frames. For $U \in \cU^k_P$ we have $V_{n-k}(P_U) > 0$.

Furthermore, we define
$$ \cU := \bigcup_{k=0}^{n-1} \cU^k $$
and 
$$ \cU_P := \bigcup_{k=0}^{n-1} \cU^k_P.$$

Consider $U=(u_1,\ldots,u_k) \in \cU^k$ and define the multiplication of $U$ with $\eta = (\epsilon_1, \ldots, \epsilon_k) \in \{ \pm 1 \}^k$ by
$\eta U:= (\epsilon_1 u_1,\ldots,\epsilon_k u_k)$.
A map $U \mapsto f_U$ is called \emph{odd} if $f_{\eta U} = \sgn(\eta)f_U$ for all such $\eta$ and $U$, where $\sgn(\eta):=\epsilon_1 \ldots \epsilon_k$.
We call $U \mapsto f_U$, $U \in \cU$, odd if every restriction to a set $\cU^k$, $k = 0, \ldots, n-1$, is odd.

  \section{Some technical results}
This section provides the necessary technical tools used later on.

\begin{lemma} \label{lemmaZ3}
  Let $k \in \{1,\ldots,n-1\}$, $P \in \bigcup \cP^k_o(V)$ and $Q \in \bigcup \cP^{n-k}_o(W)$ such that $V \oplus W = \mathbb{R}^n$.
  For all $\lambda > 0$ we have
  \begin{equation} \label{modZ3}
    \lambda P + Q \sim  P + \lambda Q \mod \cZ_3. 
  \end{equation}
\end{lemma}
\begin{proof}
  By choosing $\mu=1$ in Lemma \ref{LemmaModZ3}, we see that the statement is true for all simplices.
  Since all elements in $\bigcup \cP^k_o(V)$ respectively $\bigcup \cP^{n-k}_o(W)$
  can be decomposed into $k$-simplices respectively $(n-k)$-simplices, the statement follows.
\end{proof}

In what follows, we consider expressions of the form
$$\sum_{U \in \cU^{n-k}_P} f_U(V_k(P_U)),$$
where the $f_U$'s are additive functions.
Note that this sum is well-defined since there are only finitely many $P$-tight frames.

\begin{lemma} \label{lemmaSumVal}
  Let $1 \leq k \leq n$ and for every $U \in \cU^{n-k}$ let $f_U \colon \R \rightarrow \R$ be an additive function
  such that $U \mapsto f_U$ is odd.
  Define $\phi \colon \cP^n \to \R$ by
  \begin{equation} \label{equationVanish}
    \phi(P) = \sum_{U \in \cU^{n-k}_P} f_U(V_k(P_U))
  \end{equation}
  for all $P \in \cP^n$.
  Then $\phi$ defines a rational-$k$-homogeneous simple translation-invariant valuation on $\cP^n$.
  Moreover, $\phi$ vanishes on $(k+1)$-cylinders.
\end{lemma}
\begin{proof}
  The translation-invariance and rational-$k$-homogeneity of $\phi$ are obvious.
  Since the sum on the right hand side is empty for lower dimensional polytopes, $\phi$ is simple.

  First, we will show that $\phi$ is a valuation.
  Let $P \in \cP^n_o$ and let $H$ be a hyperplane that decomposes $P$ into
  \[
    P^+ := P \cap H^+ \in \cP^n_o \text{ and } P^- := P \cap H^- \in \cP^n_o ,
  \]
  where $H^+$ and $H^-$ are the two closed halfspaces bounded by $H$.
  By \eqref{eq: additivity} we only have to show
  \begin{equation} \label{eq: show additivity}
    \phi(P) = \phi(P^+) + \phi(P^-) .
  \end{equation}

  To this end, let $U \in \cU^{n-k}_P$.
  We distinguish two different cases.
  If $P_U$ lies completely in one of the two closed halfspaces bounded by $H$,
  then we have either that $U$ is $P^+$-tight and $P^+_U = P_U$
  or that $U$ is $P^-$-tight and $P^-_U = P_U$ but not both.
  Note that this is even true if $P_U \subseteq H$.
  Otherwise $U$ is both $P^+$-tight and $P^-$-tight
  and $P_U = P^+_U \sqcup P^-_U$ in $\aff(P_U)$.
  By the additivity of $V_k$ and $f_U$ we get
  \[
    f_U(P_U) = f_U( V_k( P^+_U ) ) + f_U( V_k( P^-_U ) ) .
  \]
  In both cases we have the same total contributions on both sides of \eqref{eq: show additivity}.

  We still need to show that all remaining terms
  on the right hand side of \eqref{eq: show additivity} cancel out.
  Let $U \in \cU^{n-k}_{P^+}$ be such that its contribution has not yet been accounted for.
  In this case $P^+_U \subseteq H$ and there exists a unique $\eta \in \{ \pm 1 \}^{n-k}$
  such that $\eta U \in \cU^{n-k}_{P^-}$ and $P^-_{\eta U} = P^+_U$.
  This $\eta$ will have exactly one entry with $-1$.
  In particular, $\sgn(\eta) = -1$.
  We conclude
  \[
    f_{\eta U}( V_k( P^-_{\eta U} ) ) = - f_U( V_k( P^+_U ) ) .
  \]
  Hence, all remaining terms on the right hand side of \eqref{eq: show additivity} sum up to $0$.

  In the last part we will show that $\phi$ vanishes on $(k+1)$-cylinders.
  For $k=n$ there is nothing to show.
  Let $k<n$ and take $P \in Z_{k+1}$, i.e.\ $P$ can be written as
  \[
    P = P_1 + \ldots + P_{k+1},
    \text{ where } P_i \subseteq V_i \text{ for } 1 \leq i \leq k+1
    \text{ and } V_1 \oplus \ldots \oplus V_{k+1}=\mathbb{R}^n.
  \]
  Consider $P_U$ for $U \in \cU^{n-k}$.
  Every component $u_j$ of $U=(u_1,\ldots,u_{n-k})$ reduces the dimension
  of exactly one of the summands $P_1,\ldots,P_{k+1}$ by one.
  Since $P \in \cP^n_o$ and $n-(n-k)=k < k+1$, at least one of these summands reduces to a point.
  Without loss of generality, let $P_1$ reduce to a point, as otherwise, we can change the order of the summation.
  Moreover, we may assume that $u_j$ reduces $P_1$ to a point.
  This means that the first component of $P_{U'}$ for $U'=(u_1,\ldots,u_{j-1})$ reduces to a line segment
  and both frames $U=(u_1,\ldots,u_{j-1},u_j,u_{j+1},\ldots,u_{n-k})$
  and $\tilde{U}=(u_1,\ldots,u_{j-1},-u_j,u_{j+1},\ldots,u_{n-k})$ are $P$-tight.
  Since $f_U$ is an odd function, we have
  \[
    f_U(V_k(P_U)) = - f_{\tilde{U}}(V_k(P_{\tilde{U}}))
  \]
  and those two values cancel out on the right hand side of \eqref{equationVanish}.
  Because $U$ was arbitrary, we get $\phi(P)=0$.
\end{proof}

\begin{remark} Since (\ref{equationVanish}) defines a simple translation-invariant valuation, which vanishes for $(k+1)$-cylinders, it actually vanishes on $\cZ_{k+1}$, since all elements in $\cZ_{k+1}$ are translative-equidecomposable to some $Q = P_1 \sqcup \ldots \sqcup P_m$ for $P_1,\ldots,P_m \in \cZ_{k+1}$.
Moreover, the fact that (\ref{equationVanish}) vanishes on $(k+1)$-cylinders can also be seen by the following result. Nevertheless, the direct proof above gives more insight into how such $\phi$ are evaluated.
\end{remark}

The next two lemmas are simple consequences of Lemma \ref{modZk}.

\begin{lemma} \label{LemmaVanish}
Any rational-$k$-homogeneous simple translation-invariant valuation $\phi$ vanishes on $\cZ_{k+1}$. 
\end{lemma}

\begin{proof}
We will show by induction that $\phi(\cZ_l) = \{0\}$ for $l$ from $n$ to $k+1$.
Assume that this holds for $l+1$, which is trivially true for $l=n$.
Take $P \in \cZ_l$ and use (\ref{equivalence}) to obtain
$$ m P \sim m^l \bullet P \mod \cZ_{l+1}.$$
By assumption, this gives $m^k \phi(P) = m^l \phi(P)$. Since $l \geq k+1$, $\phi(P)=0$.
\end{proof}

\begin{lemma}[{cf.\ \cite[Corollary 6.3.2 (b)]{Schneider2014}}] \label{LemmaZero}
Any rational-$k$-homogeneous simple transla\-tion-invariant valuation $\phi$ which vanishes on $\cZ_{k}$ is identically zero. 
\end{lemma}

\begin{proof}
By Lemma \ref{LemmaVanish}, we already know that
$$ \phi(P)= 0 \text{ for all } P \in \cZ_j,\ k<j. $$
By assumption $\phi$ also vanishes on $\cZ_k$.
We can now proceed in the same way as in Lemma \ref{LemmaVanish} to see that $\phi$ vanishes on $\cZ_j$ for $j<k$ as well. In particular, $\phi$ vanishes on $\cZ_1 = \bigcup \cP^n_o$.
\end{proof}

For $u \in \mathbb{R}^n \backslash \{0\}$, let $\overline{u}$ denote the line segment from the origin to $u$.
\begin{lemma} \label{lemmak-1homogeneous}
Let $\phi \colon \cP^n \rightarrow \mathbb{R}$ be a rational-$k$-homogeneous simple translation-invariant valuation and define 
$$\phi'(P'):=\phi(P'+ \overline{u}),$$
for all $P' \in \cP^{n-1}(H)$, where $H$ is some $(n-1)$-dimensional subspace and $u\in \mathbb R^n \setminus H$. Then $\phi'$ is a rational-$(k-1)$-homogeneous simple translation-invariant valuation.
\end{lemma}

\begin{proof}
It is clear that $\phi'$ is a simple translation-invariant valuation. We only need to prove that it is rational-$(k-1)$-homogeneous.
For this, it is enough to show $\phi'(mP')=m^{k-1}\phi'(P')$ for every integer $m >0$. We have
$$ \phi'(m P') = \phi(m P' + \overline{u}), $$
where $m P' + \overline{u}$ is a $2$-cylinder with base $m P'$ and height $|u|$. We can stack $m$ copies of this cylinder to get $m P' + m \overline{u}$, i.e.\ we have
$$ m \bullet (mP'+ \overline{u}) \sim mP' +m\overline{u} = m(P'+\overline{u}). $$
Since $\phi$ is a simple translation-invariant valuation, we conclude
$$ m \ \phi(m P' + \overline{u}) = \phi(m (P' + \overline{u})). $$
Since $\phi$ is also rational-$k$-homogeneous, we get
$$\phi'(m P') =  \frac{1}{m} \phi(m (P' + \overline{u})) = m^{k-1} \phi(P'+\overline{u}) = m^{k-1} \phi'(P'). $$
\end{proof}

  \section{Proof of the Main Results} \label{se: classification}
The main goal in this chapter is to prove that any simple translation-invariant valuation $\phi \colon \cP^n \rightarrow \R$ can be written as
$$\phi(P) = \sum_{k=1}^n \sum_{U \in \cU^{n-k}_P} f_U(V_k(P_U)),\ \forall P \in \cP^n, $$
where for all $U \in \cU$ the map $f_U \colon \R \to \R$ is an additive function such that $U \mapsto f_U$ is odd.

The idea is to first show a similar result for rational-$k$-homogeneous simple translation-invariant valuations in Theorem \ref{Th1}
and then to generalize it to all simple translation-invariant valuations in Theorem \ref{Th2}.

The following proof is heavily influenced by Hadwiger's \cite{Hadwiger1952} proof of his classification of
weakly-continuous simple translation-invariant valuations.
However, there are some key differences.
Hadwiger does not prove the explicit representation directly,
but instead proves a recursive representation first.
As a result he has to also prove an extension lemma alongside the classification using simultaneous induction.
Working directly with the frame representation removes these difficulties.
Furthermore, reducing the problem to rational-$k$-homogeneous valuations and using Lemma \ref{lemmaZ3}
make it possible to do a more careful analysis leading to the stronger result.

\begin{theorem} \label{Th1} Let $1 \leq k \leq n$ and $\phi \colon \cP^n \rightarrow \R$ be a rational-$k$-homogeneous simple translation-invariant valuation. For all $U \in \cU^{n-k}$ there exists an additive function $f_U \colon \R \to \R$ such that $U \mapsto f_U$ is odd and $\phi$ can be represented by
  \begin{equation} \label{rep1}
    \phi(P) = \sum_{U \in \cU^{n-k}_P} f_U(V_k(P_U))
  \end{equation}
for all $P \in \cP^n$.
\end{theorem}

\begin{proof} We use induction on the dimension $n$.

For $n=1$ there is only one choice for $k$, namely $k=n=1$. Any $P \in \cP_o^1$ is just a closed interval. Since $\phi$ is translation-invariant, it can only depend on the length $V_1(P)$ of $P$. By the additivity of $\phi$, this leads to $\phi(P) = f(V_1(P))$ for some additive $f$. Since on the right hand side of (\ref{rep1}) we only sum over $U \in \cU^0 = \left\{()\right\}$, we are done.

Now we consider the case of dimension $n\geq 2$. For $k \geq 2$, consider a $(k-1)$-cylinder lying in the orthogonal complement of $e_n$, so we take
  $$ P' = P_1+P_2+\ldots+P_{k-1} \in Z_{k-1}(e^\perp_n), $$
where $P_i \subseteq V_i$ for $1 \leq i \leq k-1$ and $V_1 \oplus \ldots \oplus V_{k-1}=e^\perp_n$.
For some $h>0$ we define the $k$-cylinder $P$ by
  $$ P=P'+h \overline{e_n}.$$
 Let us now apply Lemma \ref{lemmaZ3} to $P_{k-1}$ and $h^{\frac{k-2}{k-1}} \overline{e_n}$, which gives 
  $$P_1+\ldots+P_{k-2}+(P_{k-1}+h \overline{e_n}) \sim P_1+\ldots+P_{k-2}+(h^\frac{1}{k-1}P_{k-1}+h^\frac{k-2}{k-1} \overline{e_n}) \mod \cZ_{k+1}.$$
Repeating this for $P_1,\ldots,P_{k-2}$ results in
$$ P \sim h^\frac{1}{k-1} P'+\overline{e_n} \mod \cZ_{k+1}. $$
Since $\phi$ is a rational-$k$-homogeneous simple translation-invariant valuation, it vanishes on $\cZ_{k+1}$, as seen in Lemma \ref{LemmaVanish}. This leads to
\begin{equation} \label{PhiP}
\phi(P) = \phi(h^\frac{1}{k-1} P'+\overline{e_n}).
\end{equation}

Let us define
$$ \phi'_{e_n}(P'):=\phi(P'+\overline{e_n}),$$
for $P' \in \cP^{n-1}(e_n^\perp)$. By Lemma \ref{lemmak-1homogeneous} we know that $\phi'_{e_n}$ is a rational-$(k-1)$-homogeneous simple translation-invariant valuation. 

We can therefore apply our induction hypothesis to $\phi'_{e_n}$, which means that for all $U \in \cU^{(n-1)-(k-1)}(e_n^\perp)$ there exists an additive function $f'_U$ such that $U \mapsto f'_U$ is odd and for all $P' \in \cP^{n-1}(e^\perp_n)$ we have
$$ \phi'_{e_n}(P') = \sum_{U \in \cU^{n-k}_{P'}(e^\perp_n)} f'_U(V_{k-1}(P'_U)).$$

Note that for $U \in \cU^{n-k}_{P'}(e^\perp_n)$ we have $V_k((P'+ h \overline{e_n})_U)= h V_{k-1}(P'_U)$.  By the $(k-1)$-homogeneity of $V_{k-1}$, we have $V_{k-1}(h^\frac{1}{k-1} P'_U) = h V_{k-1}(P'_U)$. Using (\ref{PhiP}) leads to
\begin{align*}
 \phi(P) &= \phi(h^\frac{1}{k-1} P' + \overline{e_n})\\
 &=\phi'_{e_n}(h^\frac{1}{k-1} P') \\
 &= \sum_{U \in \cU^{n-k}_{P'}(e^\perp_n)} f'_U(h V_{k-1}(P'_U))\\
 &=\sum_{U \in \cU^{n-k}_P} f_U(V_k(P_U))
\end{align*}
for all $P = P'+h \overline{e_n}$ with $P' \in Z_{k-1}(e_n^\perp)$, where $f_U$ is given by
  $$f_U = \begin{cases} f'_U &\mbox{if } U \in \cU^{n-k}(e^\perp_n)\\
0 & \mbox{otherwise }  \end{cases}.$$
Clearly, this also holds for $P' \in \cZ_{k-1}(e_n^\perp)$.

Let us define
\begin{equation} \label{psi}
\psi(P):= \phi(P) - \sum_{U \in \cU^{n-k}_P} f_U(V_k(P_U))
\end{equation}
for all $P \in \cP^n$. Note that by Lemma \ref{lemmaSumVal} the map $\psi$ is again a rational-$k$-homogeneous simple translation-invariant valuation on $\cP^n$. By construction of $\psi$ we furthermore have
$$ \psi(P' + h \overline{e_n}) = 0$$
for all $P' \in \cZ_{k-1}(e^\perp_n)$ and $h>0$. The map $P' \mapsto \psi(P' + h \overline{e_n})$ is a rational-$(k-1)$-homogeneous simple translation-invariant valuation by Lemma \ref{lemmak-1homogeneous}. Moreover, it vanishes on $\cZ_{k-1}(e_n^\perp)$. By Lemma \ref{LemmaZero}, this map is therefore identically zero, which leads to
$$ \psi(P' + h \overline{e_n}) = 0 $$
for all $P' \in \cP^{n-1}(e_n^\perp)$.
Remember that we started this construction for $k \geq 2$. For $k=1$, we have $\phi(\cZ_2)=\{0\}$ by Lemma \ref{LemmaVanish}. Hence, we can take $\phi = \psi$ in this case and obtain $\psi(P' + h \overline{e_n}) = 0$ as well for all $P' \in \cP^{n-1}(e^\perp_n)$.

The next step is to consider skew-cylinders.
For this, take a unit vector $u \in S^{n-1}$ with $\ \left\langle  u,e_n \right\rangle > 0$. We consider $P' \in \cP^{n-1}(u^\perp)$ where we translate $P'$ such that it lies completely above $e^\perp_n$. Since $\psi$ vanishes on all cylinders of the form $P'+h \overline{e_n}$ for $P' \in \cP^{n-1}(e^\perp_n)$, the value of $\psi$ on skew-cylinders $\conv (P' \cup P'|_{e^\perp_n})$ does not depend on how we translated $P'$. Therefore,
\begin{equation} \label{eq: psi'_u}
\psi'_u(P'):=\psi(\conv (P' \cup P'|_{e^\perp_n}))
\end{equation}

is a well-defined simple translation-invariant valuation on $\cP^{n-1}(u^\perp)$. Moreover, since $\psi$ is rational-$k$-homogeneous, also $\psi'_u$ is. 

We can apply our induction hypothesis to $\psi'_{u}$. For all $U \in \cU^{(n-1)-k}(u^\perp)$ there exists an additive function $g'_{u,U}$ such that $U \mapsto g'_{u,U}$ is odd and
  $$ \psi'_{u}(P') = \sum_{U \in \cU^{(n-1)-k}_{P'}(u^\perp)} g'_{u,U}(V_k(P'_U))$$
for all $P' \in \cP^{n-1}(u^\perp)$. Let $P:=\conv (P'\cup P'|_{e_n^\perp})$ for some $P' \in \cP^{n-1}(u^\perp)$ that lies completely above $e_n^\perp$. Using \eqref{eq: psi'_u}, we obtain
\begin{align*} \label{skewCylinder}
 \psi(P)
 &= \psi'_{u}(P') \\
 &= \sum_{U \in \cU^{(n-1)-k}_{P'}(u^\perp)} g'_{u,U}(V_k(P'_U)) \\
 &= \sum_{U \in \cU^{n-k}_P} g_{u,U}(V_k(P_U)),
\end{align*}
where
$$g_{u,U} := \begin{cases} \pm g'_{\hat{U}} &\mbox{if } U=(\pm u,\hat{U}) \\
0 & \mbox{otherwise }  \end{cases}.$$

Comparing with \eqref{psi}, we obtain a representation for $\phi$ of the form
\begin{equation} \label{repSkew}
\phi(P) = \sum_{U \in \cU^{n-k}_P} h_U(V_k(P_U))
\end{equation}
for all skew-cylinders with base in $e_n^\perp$, where $h_U$ is an additive function for all $U \in \cU^{n-k}$ such that $U \mapsto h_U$ is odd.

Now we consider an arbitrary $P \in \cP^n_o$ which lies, without loss of generality, completely on one side of $e^\perp_n$.
Let $P'_i$, $i=1,\ldots, l$, be those facets whose normals have positive scalar product with $e_n$ and
$P'_i$, $i=l+1,\ldots, m$, be those facets whose normals have negative scalar product with $e_n$.
The polytope $\conv (P \cup P|_{e_n^\perp})$ can be decomposed in two ways. We have 
$$ \conv (P \cup P|_{e_n^\perp}) = P \sqcup S_1 \sqcup \ldots \sqcup S_l$$
and
$$ \conv (P \cup P|_{e_n^\perp}) = S_{l+1} \sqcup \ldots \sqcup S_m, $$
where $S_i:= \text{conv}(P'_i \cup P'_i|_{e^\perp_n})$ are skew-cylinders.
By the additivity of $\phi$, we get
$$\phi(P)= \sum_{i=l+1}^m \phi(S_i) - \sum_{i=1}^l \phi(S_i). $$
Hence, the representation \eqref{repSkew} holds for all $P \in \cP^n_o$.
\end{proof}

Now we are able to prove our main result.

\begin{theorem} \label{Th2}
  A map $\phi \colon \cP^n \rightarrow \R$ is a simple translation-invariant valuation if and only if
  for all $U \in \cU$ there exists an additive function $f_U \colon \R \to \R$ such that $U \mapsto f_U$ is odd and
  \begin{equation} \label{rep2}
    \phi(P) = \sum_{k=1}^n \sum_{U \in \cU^{n-k}_P} f_U(V_k(P_U))
  \end{equation}
  for all $P \in \cP^n$.
\end{theorem}

\begin{proof}
The right hand side of \eqref{rep2} is always a simple translation-invariant valuation by Lemma \ref{lemmaSumVal}
for any choice of $f_U \colon \R \to \R$, $U \in \cU$, as long as every $f_U$ is additive and $U \mapsto f_U$ is odd.

Assume that $\phi$ is a simple translation-invariant valuation.
By Theorem \ref{homparts}, $\phi$ can be decomposed as
$$ \phi = \sum_{k=1}^n \phi_k,$$
where $\phi_k$ is the rational-$k$-homogeneous part of $\phi$. By Theorem \ref{Th1}, for all $U \in \cU^{n-k}$ there exists an additive function $f^k_U$ such that $U \mapsto f^k_U$ is odd and
$$ \phi_k(P) = \sum_{U \in \cU^{n-k}_P} f^k_U(V_k(P_U))$$
for all $P \in \cP^n$. Hence, we have
$$ \phi(P) = \sum_{k=1}^n \phi_k(P) = \sum_{k=1}^n \sum_{U \in \cU^{n-k}_P} f_U(V_k(P_U))$$
for all $P \in \cP^n$, where $f_U := f^k_U$ for all $U \in \cU^{n-k}$ and $k=1,\ldots ,n$.
\end{proof}

The last theorem is syntactically similar to \cite[Theorem 19]{McMullen1989},
but it seems unlikely to the authors that it is easily possible to derive Theorem \ref{Th2} from it.
See, however, Remark \ref{re: classification from conditions}.

A valuation $\phi \colon \cP^n \to \R$ is called dilation-continuous if $\lambda \mapsto \phi(\lambda P)$, $\lambda > 0$, is continuous for all $P \in \cP^n$.
It is called weakly-continuous if it is continuous with respect to parallel displacements of individual facets (see \cite[p.\ 348]{Schneider2014}).
Clearly, weak-continuity implies dilation-continuity.
Also note that rational-$k$-homogeneity and dilation-continuity are equivalent to $k$-homogeneity.

The next result, originally due to Jessen and Thorup \cite{JessenThorup1978} and independently to Sah \cite{Sah1979}
(but note the comments in McMullen \cite{McMullen1983}),
can easily be proved using the same techniques as above.
Since the right hand side of \eqref{eq: representation dilation-continuous} is clearly weakly-continuous,
it implies that dilation-continuity and weak-continuity are equivalent for simple translation-invariant valuations.

\begin{theorem} \label{th: classification dilation}
  A map $\phi \colon \cP^n \to \R$ is a dilation-continuous simple translation-invariant valuation if and only if
  for all $U \in \cU$ there exists a constant $c_U \in \R$ such that $U \mapsto c_U$ is odd and
  \begin{equation} \label{eq: representation dilation-continuous}
    \phi(P) = \sum_{k=1}^n \sum_{U \in \cU^{n-k}_P} c_U V_k(P_U)
  \end{equation}
  for all $P \in \cP^n$.
\end{theorem}
\begin{proof}
  Assume that $\phi$ is a dilation-continuous simple translation-invariant valuation.
  Using Theorem \ref{homparts}, we can assume without loss of generality that $\phi$ is $k$-homogeneous.
  We will retrace the proof of Theorem \ref{Th1} without repeating it here.

  For $n=1$ we have $\phi(P) = f(V_1(P))$ for some additive $f$ for all $P \in \cP^1$.
  Since $\phi$ is $1$-homogeneous, the function $f$ has to be linear.

  For the induction part there is only one non-obvious step.
  That is, we need to show that $\phi'_{e_n}$ is $(k-1)$-homogeneous.
  First note that we have
  \begin{equation} \label{eq: lambda^{k-1} h}
    \phi(\lambda P' + h \overline{e_n})
    = \sum_{U \in \cU^{n-k}_{P'}(e_n^\perp)} f'_U(\lambda^{k-1} h V_{k-1}(P'_U))
  \end{equation}
  for all $P' \in \cP^{n-1}(e_n^\perp)$, $\lambda > 0$ and $h > 0$,
  which follows from
  \[
    \psi(P' + h \overline{e_n}) = 0
  \]
  for all $P' \in \cP^{n-1}(e_n^\perp)$ and $h > 0$.
  Hence, we calculate
  \[
    \phi'_{e_n}(\lambda P')
    = \phi(\lambda P' + \overline{e_n})
    = \phi(\lambda^{\frac {k-1} k} P' + \lambda^{\frac {k-1} k} \overline{e_n})
    = \lambda^{k-1} \phi(P' + \overline{e_n})
    = \lambda^{k-1} \phi'_{e_n}(P') ,
  \]
  where we used \eqref{eq: lambda^{k-1} h} for the second equality.
  This completes one implication of the theorem.
  The other one is obvious.
\end{proof}

  \section{Conditions for Translative-Equidecomposability} \label{se: conditions}
We are now able to give new short proofs of the well-known conditions for translative-equidecomposability.
But first, recall Theorem \ref{FormalCriterion} for $G$ the translation group.

\begin{theorem}\label{th: G-equidecomposability for G = translations}
  Two elements $P$ and $Q$ of $\bigcup \cP^n_o$ are translative-equidecomposable if and only if
  \[
    \phi(P) = \phi(Q)
  \]
  for all simple translation-invariant valuations.
\end{theorem}

Combining this with the classification of simple translation-invariant valuations obtained in Theorem \ref{Th2},
we get the following necessary and sufficient conditions for translative-equidecomposability as a simple corollary.
They were first proved by Jessen and Thorup \cite{JessenThorup1978} and independently by Sah \cite{Sah1979}.
Before that, in dimension $n=2$ respectively $n=3$ the problem of translative-equidecomposability was already solved by
Glur and Hadwiger \cite{GlurHadwiger1951} respectively Hadwiger \cite{Hadwiger1968}.
To formulate them, we define the (basic) Hadwiger functional $H_U$ for fixed $U \in \cU^{n-k}$, $k \in \{ 1, \ldots, n \}$, by
\[
  H_U(P) := \sum_{\eta \in \{ \pm 1 \}^{n-k}} \sgn (\eta) V_k \left( P_{\eta U} \right)
\]
for all $P \in \cP^n$.
By Lemma \ref{lemmaSumVal}, $H_U$ is a (rational-)$k$-homogeneous simple translation-invariant valuation.
Note that $H_{\eta U} = \sgn (\eta) H_U$ for all $\eta \in \{ \pm 1 \}^{n-k}$.

\begin{corollary} \label{co: condition for translative-equidecomposability}
  Two elements $P$ and $Q$ of $\bigcup \cP^n_o$ are translative-equidecomposable if and only if
  \begin{equation}\label{eq: condition for translative-equidecomposability}
    H_U(P) = H_U(Q)
  \end{equation}
  for all $U \in \cU$.
\end{corollary}
\begin{proof}
  Clearly, if $P$ and $Q$ are translative-equidecomposable,
  then \eqref{eq: condition for translative-equidecomposability} holds
  since $H_U$ is a simple translation-invariant valuation for all $U \in \cU$.

  Conversely, assume that \eqref{eq: condition for translative-equidecomposability} holds for all $U \in \cU$.
  Let $\phi$ be a simple translation-invariant valuation.
  By Theorem \ref{Th2} we have the following representation
  \[
    \phi(R) = \sum_{k = 1}^n \sum_{U \in \cU^{n-k}_R} f_U \left( V_k \left( R_U \right) \right)
  \]
  for all $R \in \cP^n$, where $f_U$ is an additive function for all $U \in \cU$ and $U \mapsto f_U$ is odd.
  We can rewrite this to
  \[
    \phi(R) = \sum_{k = 1}^n \frac 1 {2^{n-k}} \sum_{U \in \cU^{n-k}_R} f_U \left( H_U(R) \right) .
  \]
  Hence, $\phi(P) = \phi(Q)$.
  Using Theorem \ref{th: G-equidecomposability for G = translations},
  we see that $P$ and $Q$ are translative-equidecomposable.
\end{proof}

\begin{remark} \label{re: classification from conditions}
  As we will see below, it is also possible to prove Theorem \ref{Th2} from Corollary \ref{co: condition for translative-equidecomposability}.
  Hence, these two results are more or less equivalent.
  However, the proofs of Jessen and Thorup \cite{JessenThorup1978} respectively Sah \cite{Sah1979} of Corollary \ref{co: condition for translative-equidecomposability}
  are much longer than the proof of Theorem \ref{Th2}.

  Consider the map $\mathcal H(P) := (H_U(P))_{U \in \cU}$.
  This is a map from $\bigcup \cP^n_o$ to
  \[
    \R^{\cU*} := \{ (c_U)_{U \in \cU} \in \R^{\cU} : \text{$c_U \neq 0$ for only finitely many $U \in \cU$} \} .
  \]
  Clearly, $\mathcal H(\bigcup \cP^n_o)$ is an abelian subsemigroup of $\R^{\cU*}$.

  Let $\phi$ be a simple translation-invariant valuation.
  By Corollary \ref{co: condition for translative-equidecomposability}
  we can define an additive map $\tilde \phi$ from $\mathcal H(\bigcup \cP^n_o)$ to $\R$
  such that $\tilde \phi \circ \mathcal H  = \phi$.
  We can additively extend $\tilde \phi$ first to the abelian subgroup generated by $\mathcal H(\bigcup \cP^n_o)$
  and then, since $\R^{\cU*}$ is divisible, to the whole abelian group $\R^{\cU*}$.
  This extension is not necessarily unique.
  Nevertheless, we will also denote it by $\tilde \phi$.
  For $U \in \cU$ define
  \[
    f_U(x) := \tilde \phi( x \delta_U )
  \]
  for all $x \in \R$, where $\delta_U$ denotes the tuple that has entry $1$ at $U$ and is zero otherwise.
  Clearly, $f_U$ is additive.
  Furthermore, since $U \mapsto H_U$ is odd,
  we can choose the extension $\tilde \phi$ in such a way that $U \mapsto f_U$ is odd.
  We have
  \begin{align*}
    \phi(P)
    &= \tilde \phi( \mathcal H(P) ) \\
    &= \tilde \phi( \sum_{U \in \cU_P} H_U(P) \delta_U ) \\
    &= \sum_{U \in \cU_P} \tilde \phi( H_U(P) \delta_U ) \\
    &= \sum_{U \in \cU_P} f_U( H_U(P) ) \\
    &= \sum_{k = 1}^n 2^{n-k} \sum_{U \in \cU^{n-k}_P} f_U( V_k(P_U) ) ,
  \end{align*}
  which is what we wanted to show.
\end{remark}

  \section{Classifications of Klain and Schneider} \label{se: Klain-Schneider}
We can also use the techniques from the proof of Theorem \ref{Th1}
to \emph{simultaneously} prove classification theorems of Klain \cite{Klain1995} (see also \cite{KlainRota1997})
and Schneider \cite{Schneider1996}
for \emph{continuous} simple translation-invariant valuations.
Klain's proof of his classification of \emph{even} valuations
is already mostly elementary and our new proof will reuse some of its ideas.
However, our new proof simplifies Schneider's proof of his classification of \emph{odd} valuations considerably.
In particular, it avoids the use of triangle bodies.

The next lemma and its proof is implicitly contained in Klain \cite{Klain1995}.
We include its proof to keep the exposition self-contained.
Here, a prism is a $2$-cylinder that has some line segment as summand. 
A zonotope is a polytope which can be written as a Minkowski sum of line segments
and a zonoid is a convex body which can be approximated in the Hausdorff metric by zonotopes.
Finally, a convex body $K$ is called a generalized zonoid if there exist zonoids $Y$ and $Z$ such that
\[
  K + Y = Z .
\]

\begin{lemma}\label{le: Klain}
  Let $\phi \colon \mathcal K^n \to \R$ be a continuous simple translation-invariant valuation.
  If $\phi$ is even and vanishes on all prisms, then it vanishes everywhere.
\end{lemma}
\begin{proof}
  Let $P \in \cP^n_o$ and $u \in \R^n \setminus \{ 0 \}$.
  The Minkowski sum $P + \overline u$ can be decomposed into $P$ and prisms.
  Hence, $\phi(P + \overline u) = \phi(P)$.
  By induction we see that $\phi(P+Z) = \phi(P)$ for every zonotope $Z \in \cP^n$.
  By continuity $\phi(K + Z) = \phi(K)$ for all $K \in \cK^n$ and all zonoids $Z \in \cK^n$.
  In particular, $\phi$ vanishes on zonoids.
  From the last two facts we deduce that if $K \in \cK^n$ is a generalized zonoid, then $\phi(K) = 0$.
  Since generalized zonoids are dense in the set of centrally symmetric convex bodies
  (see e.g.\ \cite[Corollary 3.5.7]{Schneider2014}),
  $\phi$ vanishes on centrally symmetric convex bodies.

  Let $u_1, \ldots, u_n \in \R^n$ be linearly independent.
  Let $S := \conv \{ 0, u_1, \ldots, u_n \}$
  and $T$ be the parallelotope spanned by $u_1, \ldots, u_n$.
  We can decompose $T$ into $S$, $u_1 + \ldots + u_n + (-S)$ and some centrally symmetric polytope.
  Since $T$ itself is also centrally symmetric,
  $\phi$ is even and $\phi$ vanishes on centrally symmetric convex bodies, $\phi(S) = 0$.
  Therefore, $\phi$ vanishes on all $n$-simplices.
  Clearly, $\phi$ now has to vanish everywhere.
\end{proof}

We will also need the following theorem taken from McMullen \cite{McMullen1980} (compare also Sah \cite{Sah1979}).
Again, we include its proof for completeness.

\begin{theorem}\label{th: McMullen}
  Let $\phi \colon \mathcal P^n \to \R$ be a continuous simple translation-invariant valuation.
  If $\phi$ is $(n-1)$-homogeneous,
  then there exists a continuous odd function $f \colon S^{n-1} \to \R$ such that
  \begin{equation}\label{eq: (n-1)-homogeneous}
    \phi(P) = \sum_{u \in S^{n-1}} f(u) V_{n-1}(P_{(u)})
  \end{equation}
  for all $P \in \mathcal P^n$.
  Furthermore, $f$ is unique up to restrictions of linear functions.
\end{theorem}
\begin{proof}
  By Theorem \ref{th: classification dilation} there exists an odd function $f \colon S^{n-1} \to \R$
  such that \eqref{eq: (n-1)-homogeneous} holds.

  Let $u_0 \in S^{n-1}$.
  Choose a basis $b_1, \ldots, b_n \in S^{n-1}$
  such that all coordinates of $u_0$ with respect to this basis are positive.
  For $u \in S^{n-1}$ in some neighborhood of $u_0$
  let $S^u$ be a simplex with outer normals $u, b_1, \ldots, b_n$ and $V_{n-1}(S^u_{(u)}) = 1$.
  We can easily choose $S^u$ in such a way that $u \mapsto S^u$ is continuous in some neighborhood of $u_0$.
  It is well known (see e.g.\ \cite[Lemma 5.1.1]{Schneider2014})
  that if $f$ in \eqref{eq: (n-1)-homogeneous} is the restriction of a linear function,
  then the right hand side vanishes.
  Hence, we can assume without loss of generality that $f(b_i) = 0$, $i \in \{ 1, \ldots n \}$.
  We conclude $f(u) = \phi(S^u)$.
  Therefore, if $\phi$ is continuous, then $f$ is continuous in some neighborhood of $u_0$.
  Since $u_0$ was arbitrary, $f$ is continuous everywhere.
  Furthermore, if $\phi$ vanishes, then we see that $f$ vanishes up to restrictions of linear functions.
\end{proof}

In a remark in \cite{Schneider1996} Schneider asks whether Hadwiger's characterization
of weakly-continuous simple translation-invariant valuations
could be used to prove the classification in the continuous case.
In a sense the following lemma gives a positive answer.
It adapts the ideas of the proof of Theorem \ref{Th1}.

\begin{lemma}\label{le: k-homogeneous, k geq n-2}
  Let $n \geq 2$ and $\phi \colon \cK^n \to \R$ be a continuous simple translation-invariant valuation.
  If $\phi$ is $k$-homogeneous with $k \leq n-2$, then $\phi = 0$.
\end{lemma}
\begin{proof}
  We will show the result by induction.
  For $n = 2$ it follows directly from Theorem \ref{homparts} and the continuity of $\phi$.
  Assume $n \geq 3$ and that the theorem holds in dimension $n-1$.

  First assume $k \leq n-3$.
  We retrace the proof of Theorem \ref{Th1} without repeating it here.
  There we constructed a representation for $\phi$ from valuations on $\cP^{n-1}$.
  However, if $\phi$ is defined on $\cK^n$, we can also define these valuations on $\cK^{n-1}$.
  Clearly, $\phi'_{e_n}$ is continuous if $\phi$ is.
  By the induction assumption $\phi'_{e_n} = 0$.
  Hence, $\psi$ will be identical to $\phi$.
  Now, the valuations $\psi'_u$ will also be continuous.
  Again, by the induction assumption $\psi'_u = 0$.
  Therefore, $\phi$ vanishes on $\cP^n$ and by continuity also on $\cK^n$.

  If $k = n-2$, then $\phi$ is even by Theorem \ref{Th1}.
  Let $v_0 \in S^{n-1}$ and $v \in \R^n \setminus v_0^\perp$.
  Consider $P' \mapsto \phi(P' + \overline v)$, $P' \in \cK^{n-1}(v_0^\perp)$.
  Similar to above this map satisfies the induction assumption by Lemma \ref{lemmak-1homogeneous}.
  Hence, it vanishes.
  Lemma \ref{le: Klain} now implies $\phi = 0$.
\end{proof}

\begin{remark}
  It is an open problem whether or not every continuous valuation on $\cP^n$ extends to $\cK^n$.
  Hence, it would be interesting to show Theorem \ref{th: Klain-Schneider} for valuations
  that are only defined on $\cP^n$.
  The proof of Lemma \ref{le: k-homogeneous, k geq n-2} almost works for $\phi \colon \cP^n \to \R$.
  In fact, the assumption that $\phi$ is defined on $\cK^n$ is only needed to use Lemma \ref{le: Klain}.
  Note that this lemma is only applied in the case $k = n-2$.
  For $n \geq 5$ we can easily avoid using Lemma \ref{le: Klain}.
  Assume all summands of an $(n-2)$-cylinder have dimension at least $2$.
  It is easy to see that this is only possible for $n \leq 4$.
  Hence, for $n \geq 5$, every $(n-2)$-cylinder is a prism
  and we can use Lemma \ref{LemmaZero}.
  Only the cases $n = 3$, $k = 1$ and $n = 4$, $k = 2$ remain.
  These are also the only cases where the continuity assumption must be used in a non-trivial way.
  To summarize, if one could prove Lemma \ref{le: k-homogeneous, k geq n-2}
  for $\phi \colon \cP^n \to \R$ in the critical cases $n = 3$, $k = 1$ and $n = 4$, $k = 2$,
  then it could be extended to all $n \geq 2$ and $k \leq n-2$ using the arguments above.
\end{remark}

Finally, we can prove the classification result(s) of Klain \cite{Klain1995} and Schneider \cite{Schneider1996}.
Here, $S_K$ denotes the surface area measure of $K$.
For a polytope $P$ we have
\[
  S_P = \sum_{u \in S^{n-1}} V_{n-1}(P_{(u)}) \delta_u ,
\]
where $\delta_u$ denotes the Dirac measure at $u$.
Also note that $K \mapsto S_K$ is weakly-continuous in the sense of measures.

\begin{theorem} \label{th: Klain-Schneider}
  Let $n \geq 2$.
  A map $\phi \colon \cK^n \to \R$ is a continuous simple translation-invariant valuation if and only if
  there exist a constant $c \in \R$ and a continuous odd function $f \colon S^{n-1} \to \R$ such that
  \[
    \phi(K) = c V(K) + \int_{S^{n-1}} f \ dS_K
  \]
  for all $K \in \cK^n$.
  Furthermore, $f$ is unique up to restrictions of linear functions.
\end{theorem}
\begin{proof}
  One implication follows directly from Theorem \ref{homparts} together with continuity,
  Theorem \ref{th: McMullen} and Lemma \ref{le: k-homogeneous, k geq n-2}.
  The other one is obvious.
\end{proof}

  \section{Acknowledgements}
The work of the second author was supported by the ETH Zurich Postdoctoral Fellowship Program
and the Marie Curie Actions for People COFUND Program.
The authors would like to thank Monika Ludwig, Rolf Schneider and Franz Schuster
for fruitful discussions and comments.
The second author wants to especially thank Michael Eichmair and Horst Knörrer for being his mentors at ETH Zurich.

  {
    \setlength{\parskip}{0mm}
    
  }

  \bigskip \goodbreak
\begin{minipage}{0.45\textwidth}
	Katharina Kusejko \\
	ETH Zurich \\
	Department of Mathematics \\
	Rämistrasse 101\\
	8092 Zürich, Switzerland \\
	katharina.kusejko@math.ethz.ch
\end{minipage}
\begin{minipage}{0.35\textwidth}
	Lukas Parapatits \\
	ETH Zurich \\
	Department of Mathematics \\
	Rämistrasse 101 \\
	8092 Zürich, Switzerland \\
	lukas.parapatits@math.ethz.ch
\end{minipage}

\end{document}